\documentclass[a4paper,10pt]{amsart}
\usepackage[utf8]{inputenc}
\usepackage[T1]{fontenc}

\usepackage{amsmath}
\usepackage{amssymb}
\usepackage{amsthm}
\usepackage{mathtools}

\usepackage{eucal} 
\usepackage{mathdots}
\usepackage[colorlinks=true,linkcolor=red,citecolor=blue,urlcolor=green]{hyperref}
\usepackage[capitalise]{cleveref}
\usepackage{enumitem}
\usepackage{mathrsfs}  

\usepackage{tikz-cd}
\usepackage{adjustbox}
\usepackage{contour}
\usepackage{ulem}
\usepackage{faktor}

\numberwithin{equation}{section}
\crefname{equation}{}{}

\crefformat{section}{\S#2#1#3}
\crefmultiformat{section}{\S\S#2#1#3}{ and~#2#1#3}{, #2#1#3}{, and~#2#1#3}

\contourlength{0.8pt}

\newcommand{\ULL}[1]{%
  \uline{\phantom{#1}}%
  \llap{\contour{white}{#1}}%
}

\newtheorem{theorem}{Theorem}[section]
\newtheorem{prop}[theorem]{Proposition}

\newtheorem{lemma}[theorem]{Lemma}

\newtheorem{innercustomthm}{Theorem}
\newenvironment{customthm}[1]
  {\renewcommand\theinnercustomthm{#1}\innercustomthm}
  {\endinnercustomthm}

\theoremstyle{definition}

\newtheorem{rmk}[theorem]{Remark}

\newtheorem{ex}[theorem]{Example}

\DeclareMathOperator{\Hom}{\mathsf{Hom}}

\DeclareMathOperator{\Ext}{\mathsf{Ext}}
\DeclareMathOperator{\Tor}{\mathsf{Tor}}

\newcommand{\Acal}{\mathcal{A}}

\newcommand{\Ccal}{\mathcal{C}}
\newcommand{\Dcal}{\mathcal{D}}

\newcommand{\Fcal}{\mathcal{F}}
\newcommand{\Gcal}{\mathcal{G}}

\newcommand{\Pcal}{\mathcal{P}}

\newcommand{\Scal}{\mathcal{S}}
\newcommand{\Tcal}{\mathcal{T}}

\newcommand{\Xcal}{\mathcal{X}}
\newcommand{\Ycal}{\mathcal{Y}}

\newcommand{\Nbb}{\mathbb{N}}

\newcommand{\Mod}[1]{\mathsf{Mod}\mbox{-}#1}

\renewcommand{\mod}[1]{\mathsf{mod}\mbox{-}#1}
\newcommand{\modkappa}[1]{\mathsf{mod}_\kappa \mbox{-}#1}

\newcommand{\depth}{\mathsf{depth}}
\newcommand{\Kdim}{\mathsf{dim}}

\newcommand{\Spec}[1]{\mathsf{Spec}(#1)}

\newcommand*{\Perp}[1]{{}^{\perp_{#1}}}

\newcommand{\Add}{\mathsf{Add}}

\newcommand{\pp}{\mathfrak{p}}

\newcommand{\qq}{\mathfrak{q}}
\newcommand{\mm}{\mathfrak{m}}

\newcommand{\height}{\mathsf{height}}

\newcommand{{\tst}}{\textit{t}-}
\newcommand{\pd}{\mathsf{pd}}
\newcommand{\fd}{\mathsf{fd}}

\newcommand{\Findim}{\mathsf{Findim}}
\newcommand{\findim}{\mathsf{findim}}
\newcommand{\WFindim}{\mathsf{W.Findim}}

\newcommand{\Filt}{\mathsf{Filt}}

\newcommand{\CustomBig}{} % NO CHANGE TO PARENTHESES NEXT TO \QUOTMOD
\newcommand{\newterm}[1]{\ULL{\text{#1}}}
\newcommand\scalemath[2]{\scalebox{#1}{\mbox{\ensuremath{\displaystyle #2}}}}
\newcommand*{\newfaktor}[2]{% \newfaktor{#1}{#2} -> #1/#2
\scalemath{0.8}{ %SCALE FACTOR OF THE \QUOTMOD CONSTITUENTS
  \raisebox{0.5\height}{\ensuremath{#1}}% Numerator
  \mkern-5mu\diagup\mkern-4mu% Slash /
  \raisebox{-0.5\height}{\ensuremath{#2}}% Denominator
}
}

\newcommand{\quotmod}[2]{\newfaktor{#1}{#2}}

\title[The finite type and Serre's conditions]{The finite type of modules of bounded projective dimension and Serre's conditions}
\author{Michal Hrbek}
\address[M. Hrbek]{Institute of Mathematics of the Czech Academy of Sciences, \v{Z}itn\'{a} 25, 115 67 Prague, Czech Republic}
\email{hrbek@math.cas.cz}

\author{Giovanna Le Gros}
\address[G. Le Gros]{Departament de Matem\`atiques,
Universitat Aut\`onoma de Barcelona,  08193 Bellaterra
(Barcelona), Spain}
\email{giovanna.legros@uab.cat}


\thanks{The first author was supported by the GAČR project 23-05148S and the Academy of Sciences of the Czech Republic (RVO 67985840).}
\begin{document}
\begin{abstract}
Let $R$ be a commutative noetherian ring. We prove that the class of modules of projective dimension bounded by $k$ is of finite type if and only if $R$ satisfies Serre's condition $(S_k)$. In particular, this answers positively a question of Bazzoni and Herbera in the specific setting of a Gorenstein ring. Applying similar techniques, we also show that the $k$-dimensional version of the Govorov-Lazard Theorem holds if and only if $R$ satisfies the ``almost'' Serre condition $(C_{k+1})$.
\end{abstract}
\maketitle
%\tableofcontents
\section{Introduction}

For an associative unital ring, properties of classes of modules of projective dimension bounded by some positive integer provides large insight to both the module category over the ring, as well as intrinsic properties of the ring itself. This is exemplified by many classical results, for example, over a commutative noetherian ring, the relationship between the Krull dimension and the big finitistic dimension due to Bass and Raynaud-Gruson, and if the ring is moreover local, the Auslander-Buchsbaum formula. 

We are interested in when the class of modules of projective dimension at most $k$, denoted by $\Pcal_k$, can be built by certain subsets with some size restriction. In the simplest case, the class of projective modules are direct sums of countably generated ones, however, an analogous result for classes of modules of higher projective dimension does not necessarily exist in such simple terms. Nevertheless, Aldrich, Enochs, Jenda, and Oyonarte, \cite{AEJO}, showed there exists an infinite cardinal $\kappa$ that depends on the cardinality of the ring $R$, such that every module in $\Pcal_k$ is filtered by modules which are strongly $\kappa$-presented, a result which is a generalises ideas of Raynaud-Gruson. This result has important consequences on approximation theory and relative homological algebra, a theory which advanced after the introduction of cotorsion pairs. 

A cotorsion pair $(\Xcal,\Ycal)$ consists of two subcategories of $\Mod R$ maximal with respect to mutual $\Ext_R^1$-orthogonality and it is complete when its constituents give rise to left and right module approximations. The crucial result in the theory of cotorsion pairs is a theorem of Eklof and Trlifaj \cite{ET01}, which establishes the completeness of a cotorsion pair whenever its left-hand class $\Xcal$ is deconstructible: Each module from $\Xcal$ is a direct summand of a module filtered by modules of $\Xcal$ which are $\kappa$-presentable for some fixed cardinal $\kappa$. As mentioned above, in the case of $\Pcal_k$, such a cardinal, which depends on the size of the ring, was shown to exist in \cite{AEJO}.

 While Aldrich, Enochs, Jenda, and Oyonarte's result demonstrates that there exists such a cardinal $\kappa$, it does not provide a minimal one. Consequently, one of the aims of this paper is to find when, by characterising the ring, all modules in $\Pcal_k$ are filtered by its subset of strongly finitely presented modules --- the modules in $\Pcal_k$ which can be resolved by finitely generated projectives, denoted  $\Pcal_k^{<\aleph_0}$.
 
In different terminology, we say that the cotorsion pair $(\Pcal_k,\Pcal_k\Perp{})$ is of finite type if for any module $M$ we can test the Ext-orthogonality condition $\Ext_R^1(\Pcal_k,M) = 0$ just on strongly finitely presented modules $\Pcal_k^{<\aleph_0}$ of $\Pcal_k$. This question has been considered in different contexts, for example, for an $n$-dimensional noetherian ring, this is equivalent to the large and small restricted injective dimensions of Christensen, Foxby, and Frankild coinciding, see \cite{CFF02}.

 By the Eklof--Trlifaj Theorem, $(\Pcal_k,\Pcal_k\Perp{})$ being of finite type is equivalent to any module from $\Pcal_k$ being a direct summand of a module filtered by $\Pcal_k^{<\aleph_0}$. Since every projective module is a direct summand in a free one, we see immediately that $(\Pcal_0,\Pcal_0\Perp{})$ is of finite type for any ring $R$. In particular, the cotorsion pairs of finite type are exactly the tilting cotorsion pairs. 
 
 A tilting cotorsion pair $(\Acal, \Tcal)$ is a cotorsion pair such that $\Tcal$ is the first right Ext-orthogonal of a set of strongly finitely presented modules of bounded projective dimension closed under syzygies. In particular, the tilting class $\Tcal$ is closed under arbitrary direct sums. Furthermore, over commutative rings, tilting classes were classified in terms of specific finite sequences of Gabriel topologies, \cite{HS20} and \cite{AHPS14}, and over commutative noetherian rings this restricts to descending sequences of Gabriel topologies $(\Gcal_0, \dots, \Gcal_{n-1})$ where $\Gcal_i$ contains ideals of grade at least $i+1$. 
 Thus, with this correspondence, if the cotorsion pair $(\Pcal_k,\Pcal_k\Perp{})$ is of finite type, then in particular the right-hand class $\Pcal_k\Perp{}$ can be described explicitly as modules which satisfy certain divisibility conditions with respect to the ideals in the associated sequence of Gabriel topologies.

There are already some preliminary results in the literature. When $(\Pcal_1,\Pcal_1\Perp{})$ is of finite type was studied in detail by Bazzoni and Herbera in \cite{BH09}. It turns out that already this cotorsion pair might fail to be of finite type, in general. In particular, it is shown in \cite[Theorem 8.6]{BH09} that for a commutative noetherian ring $R$, $(\Pcal_1,\Pcal_1\Perp{})$ is of finite type if and only if the classical ring of quotients $Q$ of $R$ has big finitistic dimension zero. If $R$ is of Krull dimension at most one, this is equivalent to $R$ being Cohen--Macaulay. Recently, we have shown in \cite[Theorem 3.14]{HLG23} that for a commutative noetherian ring of Krull dimension $d$, $(\Pcal_d,\Pcal_d\Perp{})$ is of finite type if and only if $R$ is Cohen--Macaulay. However, whether $(\Pcal_k,\Pcal_k\Perp{})$ is of finite type for the intermediate dimensions $1<k<\Kdim(R)$, even when $R$ is a Gorenstein ring, remained unknown, \cite[Example 9.4]{BH09}. The main result of this paper answers the last question in the affirmative. In fact, we prove the following complete characterisation:
\begin{customthm}{A}\label{main-A}
	Let $R$ be a commutative noetherian ring and $k \geq 0$. Then the cotorsion pair $(\Pcal_k,\Pcal_k\Perp{})$ is of finite type if and only if $R$ satisfies Serre's condition $(S_k)$. 
	
	In particular, $(\Pcal_k,\Pcal_k\Perp{})$ is of finite type for all $k>0$ if and only if $R$ is Cohen--Macaulay.
\end{customthm}
We remark that our proof uses more elementary techniques than those used in \cite{BH09} and \cite{HLG23}, avoiding both the Mittag-Leffler condition techniques, as well as any derived or tilting theory (apart from using divisible preenvelopes). Instead, our approach relies on a combination of several kinds of commutative-algebraic reduction arguments, allowing for an induction on the parameter $k$. 

Additionally, our technique allows us to study yet another natural problem. Recall that the Govorov-Lazard Theorem asserts that, over any ring, any flat module can be written as a direct limit of finitely generated projective modules. One can wonder about a higher dimensional version: When is any module of flat dimension at most $k$ a direct limit of strongly finitely presented modules of projective dimension at most $k$? In the case of a commutative noetherian ring, we obtain the following characterisation involving the notion of almost Cohen--Macaulay rings (\cite{H98,K01,Ion}):
\begin{customthm}{B}\label{main-B}
	Let $R$ be a commutative noetherian ring and $k \geq 0$. Then $\Fcal_k = \varinjlim \Pcal_k^{<\aleph_0}$ if and only if $R$ satisfies the almost Serre's condition $(C_{k+1})$. 
	
	In particular, $\Fcal_k = \varinjlim \Pcal_k^{<\aleph_0}$ for all $k>0$ if and only if $R$ is almost Cohen--Macaulay.
\end{customthm}

%%%%%%%%%%%%%%%%%%%%%%%%%%%%%%%%%%%

\section{Preliminaries}\label{s:preliminaries}
Let $R$ be an associative unital ring. We denote by $\Mod R$ the category of all right $R$-modules and by $\mod R$ be the (full, isomorphism-closed) subcategory consisting of right $R$-modules which admit a resolution by finitely generated projective $R$-modules, which we refer to as the \newterm{strongly finitely presented modules}. Analogously, for a cardinal $\kappa$, let $\modkappa R$ denote the (full, isomorphism-closed) subcategory consisting of those modules which admit a resolution by $\kappa$-generated projective $R$-modules, which we refer to as the strongly $\kappa$-presented modules.
 For any module $M$, let $\Add(M)$ denote the subcategory consisting of all modules which are isomorphic to a direct summand of the coproduct $M^{(X)}$ for some set $X$. 
For a class $\Ccal$, let $\Ccal^\oplus$ denote the class of modules which are direct summands of a module in $\Ccal$, and $\varinjlim \Ccal$ the class of modules which are direct limits of modules in $\Ccal$. 
%Similarly, $\Prod(M)$ is the subcategory consisting of all modules which are isomorphic to a direct summand of the product $M^{X}$ for some set $X$. 

\subsection{}\label{ss:dimensions} For $n \geq 0$, let $\Pcal_n(R) = \{M \in \Mod R \mid \pd_R M \leq n\}$, and $\Fcal_n(R) = \{M \in \Mod R \mid \fd_R M \leq n\}$ denote the subcategories of $\Mod R$ consisting of all modules of projective, and respectively flat dimension bounded above by $n$. Furthermore, for a cardinal $\kappa$, we let $\Pcal_n^{\leq\kappa}(R) = \Pcal_n(R) \cap \modkappa R$, the set of at most $\kappa$-presented modules of projective dimension at most $n$. In addition, we put $\Pcal_n^{<\aleph_0}(R) = \Pcal_n(R) \cap \mod R$. We use the notation $\Pcal(R) = \bigcup_{n \geq 0}\Pcal_n(R)$ for modules of finite projective dimension, similarly we put $\Fcal(R) = \bigcup_{n \geq 0}\Fcal_n(R)$. We often omit the reference to the ring when the ring is clear from the context and write simply $\Pcal_n$, $\Fcal_n$, $\Pcal$, $\Fcal$, $\Pcal_n^{\leq\kappa}$, or $\Pcal_n^{<\aleph_0}$.

\subsection{}\label{ss:cotpairs} Given a subcategory $\Ccal$ of $\Mod R$ we use the notation $\Ccal\Perp{1}= \{M \in \Mod R \mid \Ext_R^1(C,M) = 0, ~\forall C \in \Ccal\}$ and $\Ccal\Perp{}= \{M \in \Mod R \mid \Ext_R^i(C,M) = 0 ~\forall C \in \Ccal, ~\forall i>0\}$; we also define $\Perp{1}\Ccal$ and $\Perp{}\Ccal$ analogously and we omit the curly brackets whenever $\Ccal = \{C\}$ for a module $C$. A \newterm{cotorsion pair} in $\Mod R$ is a pair $(\Xcal,\Ycal)$ of subcategories of $\Mod R$ such that $\Ycal = \Xcal\Perp{1}$ and $\Xcal = \Perp{1}\Ycal$. A cotorsion pair is \newterm{complete} if for any $M \in \Mod R$ there are exact sequences: 
$$0 \to M \to Y^M \to X^M \to 0$$ 
and 
$$0 \to Y_M \to X_M \to M \to 0$$ 
with $X^M,X_M \in \Xcal$ and $Y^M,Y_M \in \Ycal$. The first exact sequence (or, the map $M \to Y^M$) is called a \newterm{special $\Ycal$-preenvelope}, while dually we talk about \newterm{special $\Xcal$-precovers}.
The cotorsion pair is \newterm{hereditary} if $\Ext_R^i(X,Y) = 0$ for all $X \in \Xcal$, $Y \in \Ycal$, and $i>0$. Finally, a cotorsion pair is \newterm{of finite type} if there is $n \geq 0$ and a subset $\Scal$ of $\Pcal_n^{<\aleph_0}$ such that $\Ycal = \Scal\Perp{}$.

\subsection{}\label{ss:filtrations} Let $\mu$ be an ordinal, and consider a collection of submodules of an $R$-module $M$, that is $\{M_\alpha \mid M_\alpha \subseteq M, \alpha \leq \mu\}$ which are ordered under inclusion. If $M_0 = 0$, and $M_\beta = \bigcup_{\alpha<\beta}M_\alpha$ for all limit ordinals $\beta \leq \mu$, this collection of submodules is called a \newterm{continuous chain} of modules. Given a class of $R$-modules $\Ccal$, an $R$-module $M$ is said to be $\Ccal$-filtered if it has a continuous chain of submodules $\{M_\alpha \mid M_\alpha \subseteq M, \alpha \leq \mu\}$ with $M_\mu = M$ for an ordinal $\mu$, such that for every successor ordinal $\alpha \leq\mu$, $\quotmod{M_{\alpha+1}}{M_{\alpha}} \in \Ccal$. The continuous chain $\{M_\alpha \mid M_\alpha \subseteq M, \alpha \leq \mu\}$ is called a \newterm{$\Ccal$-filtration} of $M$.
The class of modules which have a filtration by modules in $\Ccal$ are denoted $\Filt(\Ccal)$.

\subsection{}\label{ss:compcotpairs} Any class $\Ccal$ generates a cotorsion pair, that is $({}^{\perp_1}(\Ccal^{\perp_1}),\Ccal^{\perp_1})$ is a cotorsion pair. If there exists a set $\Scal$ of modules which generates a cotorsion pair $(\Xcal, \Ycal)$, then the cotorsion pair is complete, and moreover $\Xcal$ is exactly the class of direct summands of $\{R\}\cup\Scal$-filtered modules, $\Filt(\{R\}\cup\Scal)^\oplus$. This follows from the Eklof Lemma, and results of Eklof and Trlifaj, \cite{ET01}, or see \cite[\S 8]{GT12} for more details.

\subsection{}\label{ss:finitetype} Suppose $R$ is $\kappa$-noetherian for some cardinal $\kappa$, that is, every right ideal of $R$ is at most $\kappa$-generated, and fix $n \in \Nbb$. Then the class $\Pcal_n$ is $\Pcal_n^{\leq \kappa}$-filtered, or $\Pcal_n = \Filt(\Pcal_n^{\leq \kappa})^\oplus$ \cite[\S 8]{GT12} or \cite{AEJO}. 
In particular, over a noetherian ring, the class $\Pcal_n$ is $\Pcal_n^{\leq \aleph_0}$-filtered, which is a result due to Raynaud and Gruson, \cite[Corollaire (3.2.5)]{RG71}. 

We are particularly interested in when the class $\Pcal_n$ is of finite type, that is when every module in $\Pcal_n$ is a direct summand of a $\Pcal_n^{<\aleph_0}$-filtered module, or equivalently when the cotorsion pairs $(\Pcal_n, \Pcal_n^\perp)$ and $({}^\perp((\Pcal_n^{<\aleph_0})^\perp),(\Pcal_n^{<\aleph_0})^\perp)$ coincide. We will denote the latter cotorsion pair by $(\Acal_n, \Tcal_n)$, a notation that arises from the fact that $\Tcal_n$ is the minimal $n$-tilting class.

More generally, the class $\Filt(\Pcal_n^{\leq \kappa})^\oplus$ always forms the left-hand class of a cotorsion pair for any cardinal $\kappa$. Additionally, the $n$-tilting classes are exactly the right-hand classes in cotorsion pairs generated by strongly finitely presented modules of projective dimension at most $n$, see \cite{BH08, BS07} or \cite[\S 13]{GT12}. 

\subsection{}\label{ss:injlim} For a class $\Ccal$, $\varinjlim \Ccal$ is not necessarily closed under direct limits. However, if $\Ccal$ is a set consisting of finitely presented modules which are closed under finite direct sums, then $\varinjlim \Ccal$ is closed under direct limits, \cite{Lenz83}. Moreover, if $\Ccal$ consists of strongly finitely presented modules, there is a cotorsion pair $(\varinjlim \Ccal, (\varinjlim \Ccal)^\perp)$ \cite[Corollary 2.4]{AHT04}.

\subsection{}\label{ss:findim} Recall that the \newterm{finitistic dimension} of $R$ is defined as 
$$\Findim(R) = \sup \{\pd_R M \mid M \in \Pcal\},$$
while the \newterm{small finitistic dimension} of $R$ is 
$$\findim(R) = \sup \{\pd_R M \mid M \in \Pcal^{<\aleph_0}\}.$$
In the last section, we will also use the \newterm{weak finitistic dimension} of $R$ defined as
$$\WFindim(R) = \sup \{\fd_R M \mid M \in \Fcal\}.$$

\subsection*{Commutative rings} From now on, let $R$ be a commutative ring.

\subsection{}\label{ss:notationmultsubsets} Let $R$ be a commutative ring and $s \in R$ and $M$ and $R$-module. We will say that $s$ is \newterm{regular} if it does not divide zero, and we will let $\Sigma \subset R$ denote the multiplicative subset of regular elements of $R$. 
The \newterm{$s$-socle} of $M$, denoted $M[s]$, is the submodule $\{m \in M \mid sm=0\}$ of $M$. It is straightforward to see that for a regular element $s \in R$, $M[s] \cong \Tor_1^R(M, \quotmod{R}{sR}) \cong \Hom_R(\quotmod{R}{sR}, M)$ by applying the appropriate functors to the short exact sequence 

\[
\begin{tikzcd}%[cramped, sep=small]
0 \arrow[r] &R \arrow[r,"{(-)\cdot s}"] &  R \arrow[r] & \quotmod{R}{sR} \arrow[r] &0.
\end{tikzcd}
\]

Consider a multiplicative subset $S$ of $R$ (of not-necessarily regular elements). The class of $S$-divisible modules, denoted $\Dcal_S$, is defined as $\{D \in \Mod R \mid \forall s \in S, sD = D  \}$. We also define the class of $S$-torsion modules $\{M \in \Mod R \mid \forall m \in M, \exists s \in S\text{ such that } ms = 0 \}$. Therefore, a module is $S$-torsion if and only if $M = \bigcup_{s \in S}M[s]$. For a module $M$, we denote the maximal $S$-torsion submodule of $M$ as $\Gamma_S(M)$. 

Additionally, recall that the module $R[S^{-1}]$ is a direct limit of copies of $R$, and $\quotmod{R[S^{-1}]}{R}$ can be seen as the direct limit of modules of the form $R/sR$ for $s \in S$, see \cite[Example IX.1.2]{SS75}.
If moreover $S$ is a multiplicative subset of regular elements, then $\Gamma_S(M) \cong \Tor^R_1(M, \quotmod{R[S^{-1}]}{R})$.

%Similarly, for a Gabriel filter $\Gcal$ of $R$, we let $\Gamma_\Gcal$ denote the $\Gcal$-torsion functor of $M$, see \cite[\S VI]{SS75}. 

\subsection{}\label{ss:multsubsets} We will be interested in a particular collection of complete cotorsion pairs which arise from particular multiplicative subsets of $R$. Let $R$ be a commutative ring and consider any countably generated multiplicative subset $S$ of $R$. Then there is a classical construction which allows one to conclude that the localisation of $R$ at $S$ has projective dimension at most one over $R$, that is $\pd_R R[S^{-1}] \leq 1$.

 This can be seen using the direct limit presentation of $R[S^{-1}]$. That is, if $S=\{s_1, s_2, \dots \}$ then we can see $R[S^{-1}]$ as the direct limit of homomorphisms
 \begin{equation}\label{eq:RS_colim}
R[S^{-1}]=  \varinjlim( R \overset{s_1}\to R \overset{s_1s_2}\to \cdots ),
 \end{equation} 
and the direct limit presentation of $R[S^{-1}]$ is 
\[
0\to \bigoplus_{n\in \Nbb}R\overset{\phi}\to \bigoplus_{n\in \Nbb}R\to R[S^{-1}]\to 0\]
where if $\{e_i\}_{i \in \Nbb}$ and $\{f_i\}_{i \in \Nbb}$ are basis elements of the first two terms, $\phi (e_i) = f_i -s_1 \cdots s_if_{i+1}$.

\subsection{}\label{ss:regmultsubsets}
If the countably generated multiplicative subset $S$ in \cref{ss:multsubsets} is made up of regular elements of $R$, then we can say more. In particular interested in the complete hereditary cotorsion pair generated by $\{\quotmod{R}{sR}\}_{s \in S}$, which we will denote $(\Acal_S, \Dcal_S)$. The class $\Dcal_S$ is exactly the class of modules divisible by $s$ for every $s \in S$, and $\Acal_S$ is the class of modules which direct summands of modules filtered by the set  $\{R\}\cup\{\quotmod{R}{sR}\}_{s \in S}$. This cotorsion pair is of finite type, and $\Acal_S \subseteq \Pcal_1$ since if $s \in S$ is a regular element, then $\pd_R \quotmod{R}{sR} \leq 1$.

Furthermore, it is straightforward to see that $R[S^{-1}] \in \Acal_S$ as each of the homomorphisms in the direct limit \cref{eq:RS_colim} are monomorphisms with cokernels in $\{\quotmod{R}{sR}\}_{s \in S}$. Therefore, the direct limit forms a continuous filtration of $R[S^{-1}]$ by the set  $\{R\}\cup\{\quotmod{R}{sR}\}_{s \in S}$. 

The cotorsion pair $(\Acal_S, \Dcal_S)$ is complete, therefore $\Acal_S$ is a special precovering class and $\Dcal_S$ is a special preenveloping class. In particular, there is the following special $\Dcal_S$-preenvelope of $R$.
\[
0 \to R \to R[S^{-1}] \to \quotmod{R[S^{-1}]}{R} \to 0
\]
In particular, $R[S^{-1}] \oplus \quotmod{R[S^{-1}]}{R}$ is a $1$-tilting module and $\Dcal_S$ its associated $1$-tilting class, see \cite{AHHT05}.

\subsection*{Commutative noetherian rings} From now on, let $R$ be a commutative noetherian ring.

\subsection{}\label{ss:RG} By classical results of Bass \cite{B62} and Raynaud-Gruson \cite{RG71}, $\Findim(R) = \Kdim(R)$, \cite[Théor\`{e}me 3.2.6]{RG71}. Assume now that $\Kdim(R)<\infty$. Then any flat $R$-module belongs to $\Pcal$ \cite{Jen70}, \cite[Corollaire 3.2.7]{RG71}, and therefore $\Pcal$ coincides with the class $\Fcal$ of all modules of finite flat dimension. It follows that $\Pcal$ can be described as the class of all modules of flat dimension bounded by $\Kdim(R)$. In symbols, we have $\Pcal_{\Kdim(R)} = \Pcal = \Fcal = \Fcal_{\Kdim(R)}$. 

If $R$ is local with maximal ideal $\mm$ then $\findim(R) = \depth(R)$ by the Auslander-Buchsbaum formula, where $\depth(R) = \min \{i \mid \Ext_R^i(\quotmod{R}{\mm},R) \neq 0 \}$.

Finally, for any $R$ we have the identity $\WFindim(R) = \sup \{\depth R_\pp \mid \pp \in \Spec R\}$, \cite[Proposition 5.1]{B62}.
%Finally, $R$ being noetherian ensures that $\Fcal_{\Kdim(R)}$ is a \newterm{definable class}, that is, a subcategory of $\Mod R$ closed under direct limits, products, and pure submodules.

\subsection{}\label{ss:serre} The following term was introduced by Serre in his characterisation of normal rings. For $k \in \Nbb$, a ring $R$ has \newterm{Serre's condition $(S_k)$}, if for every prime ideal $\pp $ of $R$, $\depth (R_\pp) \geq \min \{\height (\pp), k\}$. This condition has the property that $(S_{k+1})$ implies $(S_{k})$ for every $k>0$.
In particular, a ring is Cohen--Macaulay if and only if $(S_k)$ holds for every $k \in \Nbb$. All commutative noetherian rings satisfy $(S_0)$, and a commutative noetherian ring satisfies $(S_1)$ if and only if its classical ring of quotients is artinian, or equivalently if all associated prime ideals of $R$ are minimal primes.

\subsection{}\label{ss:ACM-cond} An analogous condition to Serre's condition $(S_k)$ to characterise almost Cohen--Macaulay rings was introduced by Ionescu \cite{Ion}, see also a more general version of Holmes \cite{H19}. Recall that a local ring $R$ is almost Cohen--Macaulay if $\Kdim(R)-1\leq \depth(R) \leq \Kdim(R)$, and for a non-local ring the condition is defined locally. 

Following \cite{Ion}, a ring $R$ satisfies condition $(C_k)$ if for every prime $\pp$ of $R$, $\depth (R_\pp) \geq \min \{\height (\pp), k\}-1$. Moreover, it is shown that a ring $R$ has $(C_k)$ for every $k \in \Nbb$ if and only if $R$ is almost Cohen--Macaulay, that $(C_{k+1})$ implies $(C_{k})$, and additionally that $(S_{k})$ implies $(C_{k})$.

The following lemma is a straightforward consequence of well-known properties of the depth and grade. We include it here for completeness. 

\begin{lemma}\label{serrequotientring}
Let $R$ be a commutative noetherian ring and $a \in R$ a regular non-unit element of $R$. 
\begin{itemize}
	\item[(i)] If $R$ satisfies $(S_k)$ for a positive $k \in \Nbb$,  $\quotmod{R}{aR}$ satisfies $(S_{k-1})$.
	\item[(ii)] If $R$ satisfies $(C_k)$ for a positive $k \in \Nbb$,  $\quotmod{R}{aR}$ satisfies $(C_{k-1})$.
\end{itemize}
\end{lemma}
\begin{proof}
(i) Take a prime ideal $\pp$ which contains $a$. We wish to show that 
\[
\depth (\CustomBig(\quotmod{R}{aR}\CustomBig)_{\quotmod{\pp}{aR}} )\geq \min\{\height \CustomBig(\quotmod{\pp}{aR}\CustomBig), k-1\}.
\]
 By assumption we know that $\depth (R_{\pp}) \geq \min\{\height (\pp), k\}$. We know that $\depth ((\quotmod{R}{aR})_{\quotmod{\pp}{aR}} )  = \depth  (R_{\pp})-1$, see \cite[Proposition 1.2.10(d)]{BH98}, and $\height (\pp)-1 = \height (\quotmod{\pp}{aR})$ by Krull's principal ideal theorem.  
Therefore, 
\begin{align*}
\depth ( \CustomBig(\quotmod{R}{aR}\CustomBig)_{\quotmod{\pp}{aR}} )&= \depth  (R_{\pp})-1\\
					&  \geq \min\{\height (\pp), k\}-1 \\
					&=\min\{\height (\quotmod{\pp}{aR}), k-1\} ,
\end{align*}
as required.

(ii) The proof is identical to (i).

\end{proof}

The following proposition will be useful when generalising from the case that $Q$, the classical quotient ring of $R$, is not countably generated. The proposition comes essentially from \cite[Proposition 2.2]{HPS}, but our hypotheses are a little more general so we include a proof for completeness. 

\begin{prop}\cite[Proposition 2.2]{HPS}\label{countsubset} Let $R$ be a commutative ring and $M$ an infinitely generated and strongly countably presented
$R$-module of finite projective dimension. Let $T \subseteq R$ be a multiplicative subset such that $M[T^{-1}]$ is a projective $R[T^{-1}]$-module. Then there is a countable multiplicative subset $S \subseteq T$ such that $M[S^{-1}]$ is a projective $R[S^{-1}]$-module.
\end{prop}

\begin{proof}
Let $M$ be a strongly countably presented module of projective dimension $k$. In the case that $M$ is infinitely generated, without loss of generality, we can assume that $M$ has a resolution of length $k$ of countably generated free modules using the Eilenberg swindle. 
%If $M$ is not infinitely generated, one can construct a finite free resolution \cite[Lemma 6.4]{BH09} of $M\oplus P$ for some finitely generated projective module $P$.

Given this free resolution, the case of $k = 1$ follows directly from the proof of \cite[Proposition 2.2]{HPS}. We use induction to extend to higher dimensions.

Let $M$ be a strongly countably presented module of projective dimension $k$ with first syzygy $\Omega_1(M)$, and suppose the statement holds for all natural numbers $i<k$. Then there exists a countably generated multiplicative subset $S$ such that $\Omega_1(M) \otimes_RR[S^{-1}] \in \Pcal_0(R[S^{-1}])$ by the induction hypothesis. Applying $(- \otimes_RR[S^{-1}])$ to the free presentation of $M$, we find an exact sequence
\[
\begin{tikzcd}%[cramped, sep=small]
0 \arrow[r] &\Omega_1(M) \otimes_RR[S^{-1}] \arrow[r] &  R[S^{-1}]^{(\alpha)} \arrow[r] &M\otimes_RR[S^{-1}] \arrow[r] &0,
\end{tikzcd}
\]
so $M\otimes_RR[S^{-1}]$ is a countably presented module of projective dimension $1$ in $\Mod R[S^{-1}] $. Using the Eilenberg swindle again and the proof of  \cite[Proposition 2.2]{HPS}, we find a countable multiplicative subset $U$ of $R[S^{-1}]$ such that $(M\otimes_RR[S^{-1}])\otimes_{R[S^{-1}]}R[S^{-1}][U^{-1}] $. The multiplicative subset generated by $U\cup S$ considered as a multiplicative subset of $R$ is countable, as required.
\end{proof}

\section{Finite type of $(\Pcal_k, \Pcal_k^\perp)$ and Serre's condition $(S_k)$}\label{s:ftserre}
In this section, we prove \cref{main-A}. We note that the initial theorem does not require noetherianity.

\begin{lemma} \label{suffdiv}
Let $R$ be a commutative ring, $S$ a countably generated multiplicative subset of $R$ consisting of regular elements of $R$, $\Dcal_S$ the class of $S$-divisible modules, and $n \in \Nbb$. Let $M\in \Pcal_n$ and fix a special $\Dcal_S$-preenvelope of $M$, $0 \to M \to D \to A \to 0$. Then $M \in \Acal_n$ if and only if $D \in \Acal_n$. 

Moreover, if $M$ is $\kappa$-generated for some infinite cardinal $\kappa$, there exists a special $\Dcal_S$-preenvelope of $M$ which is $\kappa$-generated.

\end{lemma}
\begin{proof}
The first statement holds since $A \in \Acal_S$, and every module in $\Acal_S$ is a direct summand in a module filtered by modules in $\Pcal_1^{<\aleph_0}$, see \cref{ss:regmultsubsets}. 

For the second statement, take a free presentation of $M$, $R^{(\kappa)} \overset{\phi}\to M \to 0$  the special $\Dcal_S$-envelope of $R^{(\kappa)}$, $\psi$, and we form the pushout of $\psi$ and $\phi$, which we denote $D$.  
\[
\begin{tikzcd}
0 \arrow[r] & R^{(\kappa)} \arrow[r, "\psi"] \arrow[d, two heads,  "\phi"] & R[S^{-1}]^{(\kappa)} \arrow[d, two heads] \arrow[r] & \CustomBig(\quotmod{R[S^{-1}]}{R}\CustomBig)^{(\kappa)} \arrow[d, equal] \arrow[r] & 0 \\
0 \arrow[r] & M \arrow[r, "\mu"] &  D \arrow[r] &   \CustomBig(\quotmod{R[S^{-1}]}{R}\CustomBig)^{(\kappa)} \arrow[r] & 0
\end{tikzcd}
\]
Then $\mu \colon M \to D$ is a special $\Dcal_S$-preenvelope. As both $M$ and $\CustomBig(\quotmod{R[S^{-1}]}{R}\CustomBig)^{(\kappa)}$ are at most $\kappa$-generated and $\quotmod{R[S^{-1}]}{R}$ is countably generated since $S$ is countably generated, see \cref{ss:multsubsets}, also $D$ is at most $\kappa$-generated.
\end{proof}

\begin{lemma}\label{indstepSk}
Let $R$ be a commutative noetherian ring and $S$ a countable multiplicative set consisting of regular elements of $R$. Take $M \in \Pcal_k(R)$ to be an $S$-divisible module. 
\begin{itemize}
\item[(i)] $M[s] \in \Pcal_{k-1}(\quotmod{R}{sR})$ for each $s \in S$,
\item[(ii)] $\Gamma_S(M)$ is filtered by modules in $\bigcup_{s \in S}\Pcal_{k-1}(\quotmod{R}{sR})$.%$\{P[s]\}_{s \in S}$
\end{itemize}

%Then there is a countable filtration of the $S$-torsion submodule $\Gamma_S(P)$ of $P$, $\Gamma_S(P)$, by modules $\{P[s] \otimes_RR/sR\}_{s \in S}$ and $\Omega_1(P) \otimes_RR/sR \in \Pcal_{k-1}(R/sR)$. 
\end{lemma}

\begin{proof}
Let $M \in \Pcal_k$ be an $S$-divisible module and $P^\bullet \to M \to 0$  and 
%$\begin{tikzcd}[cramped, sep=small]
%P^\bullet \arrow[r] & M \arrow[r] & 0
%\end{tikzcd}$
a projective presentation of $M$ of length $k$. We fix the first syzygy $\Omega_1(M) = \mathsf{Im}(P_1 \to P_0)$ with respect to this presentation.

(i) Fix $s \in S$. Apply the functor $(- \otimes_R \quotmod{R}{sR})$ to $P^\bullet \to M \to 0$. As $\pd_R(\quotmod{R}{sR}) \leq 1$, $\Tor_i^R(M, \quotmod{R}{sR}) = 0$ for all $i>1$, and $\Tor_1^R(\Omega_1(M), \quotmod{R}{sR})=0$ as $\Omega_1(M)$ is $S$-torsion-free, being a submodule of an $R$-free module. Additionally, by the assumption that $M$ is $S$-divisible, $M \otimes_R \quotmod{R}{sR}=0$. Therefore we find the following two exact sequences:

\[
\begin{tikzcd}[cramped, sep=small]
0 \arrow[r] & P_k \otimes_R \quotmod{R}{sR} \arrow[r] &  \cdots \arrow[r] &   P_1 \otimes_R \quotmod{R}{sR} \arrow[r] &\Omega_1(M) \otimes_R \quotmod{R}{sR} \arrow[r] &0,
\end{tikzcd}\\
\]
\[
\begin{tikzcd}[cramped, sep=small]
0 \arrow[r] & \Tor_1^R(M, \quotmod{R}{sR}) \ar[r]& \Omega_1(M) \otimes_R \quotmod{R}{sR} \arrow[r] & P_0 \otimes_R \quotmod{R}{sR} \arrow[r] &  0,
\end{tikzcd}
\]
which are also exact sequences as $\quotmod{R}{sR}$-modules. 

The first exact sequence is a $\Pcal_0\big(\quotmod{R}{sR}\big)$-presentation of $\Omega_1(M) \otimes_R \quotmod{R}{sR}$, so $\Omega_1(M) \otimes_R \quotmod{R}{sR} \in \Pcal_{k-1}(\quotmod{R}{sR})$. 
Additionally, the second exact sequence splits as $P_0 \otimes_R \quotmod{R}{sR}$ is projective in $\Mod \quotmod{R}{sR}$, therefore also $M[s] \cong \Tor_1^R(M, \quotmod{R}{sR}) \in \Pcal_{k-1}(\quotmod{R}{sR})$.

(ii) Fix an enumeration $S = \{s_i \mid i \in \Nbb\}$. By the discussion in \cref{ss:regmultsubsets}, $\quotmod{R[S^{-1}]}{R}$ is the union of the continuous filtration 
%$\{\quotmod{R}{t_iR} \overset{(-)\cdot s_{i+1}}\to \quotmod{R}{t_{i+1}R}\}_{s_j \in S}$ where $t_j \vcentcolon = s_js_{j-1}\cdots s_1$
$\begin{tikzcd}%[cramped]
\big\{\quotmod{R}{t_iR} \arrow[r, "(-)\cdot s_{i+1}"] & \quotmod{R}{t_{i+1}R}\big\}_{s_j \in S}
\end{tikzcd}$
 and 

\begin{equation}\label{eq:torsionses}
\begin{tikzcd}%[cramped, sep=small]
0 \arrow[r] &\quotmod{R}{t_iR} \arrow[r, "(-)\cdot s_{i+1}"] &   \quotmod{R}{t_{i+1}R} \arrow[r] & \quotmod{R}{s_{i+1}R} \arrow[r] &0
\end{tikzcd}
\end{equation}
is a short exact sequence. 

We apply $\Tor_1^R(M,-)$ to this filtration, which we claim yields another filtration of $\Tor_1^R(M, \quotmod{R[S^{-1}]}{R})$. Explicitly, applying $(M \otimes_R-)$ to \cref{eq:torsionses}, and noting that $\Tor_2^R(M, \quotmod{R}{s_iR})=0=M \otimes_R \quotmod{R}{t_iR}$ as $\pd_R \quotmod{R}{s_iR} \leq 1$ and $M$ is $S$-divisible, we find the short exact sequences $0 \to \Tor_1^R(M,\quotmod{R}{t_iR}) \overset{s_{i+1}}\to \Tor_1^R(M,\quotmod{R}{t_{i+1}R}) \to \Tor_1^R(M,\quotmod{R}{s_iR}) \to 0$. Thus the monomorphisms $\{\Tor_1^R(M,\quotmod{R}{t_iR}) \overset{s_{i+1}}\to \Tor_1^R(M,\quotmod{R}{t_{i+1}R})\}$ form a continuous chain of submodules of $\Tor_1^R(M,\quotmod{R[S^{-1}]}{R})$, and the cokernels are exactly $\Tor_1^R(M,\quotmod{R}{s_iR}) \cong M[s_i]$ which are in $\Pcal_{k-1}(\quotmod{R}{sR})$ by (i). 
As $\Gamma_S(M) \cong \Tor_1^R(M, \quotmod{R[S^{-1}]}{R})$, the conclusion follows. 
\end{proof}

\begin{prop}\label{Skimpliesft}
Let $R$ be a commutative noetherian ring that satisfies $(S_k)$ and let $M$ be a countably generated module in $\Pcal_k$. Then $M \in \Acal_k$, so $(\Pcal_k,\Pcal_k\Perp{})$ is of finite type.
\end{prop}
\begin{proof}
The proof will be by induction on $k$. For $k =0$, $(S_0)$ always holds and $\Pcal_0$ is always of finite type over any ring so there is nothing to prove. Note that the assertion about $M$ is sufficient to imply the finite type of $(\Pcal_k,\Pcal_k\Perp{})$, see \cref{ss:finitetype}.

Take $M$ as in the statement of the proposition, and suppose that the statement holds for $k-1$. If $(S_k)$ holds and $k>0$, then in particular the classical ring of quotients $Q$ is artinian, and so $M \otimes_RQ$ is a projective $Q$-module because $\Findim(Q)=0$. Therefore, by \cref{countsubset} there exists a countable multiplicative subset $S$ consisting of regular elements such that $M\otimes_RR[S^{-1}]$ is a projective $R[S^{-1}]$-module. 

By \cref{suffdiv}, we can assume that $M$ is a countably generated $S$-divisible module in $\Pcal_k$, and such that $M\otimes_RR[S^{-1}]$ is a projective $R[S^{-1}]$-module. Consider the following exact sequence, where $M\otimes_R \quotmod{R[S^{-1}]}{R}=0 $ by the $S$-divisiblity of $M$.

\begin{equation}\label{eq:torsionP}
\begin{tikzcd}%[cramped, sep=small]
0 \arrow[r] &\Gamma_S(M) \arrow[r] &   M \arrow[r] &M\otimes_RR[S^{-1}] \arrow[r] &0
\end{tikzcd}
\end{equation}

As $M\otimes_RR[S^{-1}] \in \Acal_1$, it is a projective $R[S^{-1}]$-module. So, to conclude that $M \in \Acal_k$, it is sufficient to show that $\Gamma_S(M) \in \Acal_k$.

By \cref{indstepSk}, $\Gamma_S(M)$ is filtered by modules in $\bigcup_{s \in S}\Pcal_{k-1}(\quotmod{R}{sR})$. Furthermore, by \cref{serrequotientring} and the induction hypothesis, $\Pcal_{k-1}(\quotmod{R}{sR})$ is of finite type in $\Mod \quotmod{R}{sR}$, and is therefore a direct summand in a $\Pcal_{k-1}^{<\aleph_0}(\quotmod{R}{sR})$-filtered $\quotmod{R}{sR}$-module. By the first change of rings theorem (see \cite[Theorem 4.3.3]{Wei94}), $\Pcal_{k-1}^{<\aleph_0}(\quotmod{R}{sR})\subseteq\Pcal^{<\aleph_0}_k(R)$ for every $s \in S$, so we conclude that $M[s] \in \Acal_k$ and therefore also $\Gamma_S(M) \in \Acal_k$, as required.

\end{proof}

We now prove the converse implication.

\begin{prop}\label{ftimpliesSk}
Let $R$ be a commutative noetherian ring such that $(\Pcal_k,\Pcal_k\Perp{})$ is of finite type for some natural number $k$. Then Serre's condition $(S_k)$ holds.
\end{prop}

\begin{proof}
We prove the contrapositive using a careful adaptation of a construction of Bass, \cite[Proposition 5.2]{B62}. Suppose $(S_k)$ does not hold, that is, there exists a prime ideal $\pp$ of $R$ such that $\depth(R_\pp) < \min\{\height(\pp), k\}$, and suppose for the contradiction that $\Pcal_k$ is of finite type.

Take $n\vcentcolon= \min\{\height(\pp), k\}$. By \cite[Proposition 5.2]{B62}, there exists a prime $\qq\subsetneq \pp$, a sequence $\mathbf{a}\vcentcolon=(a_1, \dots, a_{n-1})$ of elements in $\qq$, and $s \in R$ such that $\qq \subsetneq (\qq, s) \subsetneq \pp$, and $\mathbf{a}$ is a regular sequence in $R_\pp[s^{-1}]$. By \cite[Proposition 5.4]{B62}, we have $\pd_{R_\pp} \CustomBig(\quotmod{R_\pp[s^{-1}]}{\mathbf{a}R_\pp[s^{-1}]}\CustomBig) =n$. Put $T = R \setminus \pp$, and $M = \quotmod{R[s^{-1}]}{\mathbf{a}R[s^{-1}]}$. Since $\mathbf{a}$ is a regular sequence in $R_\pp[s^{-1}]=R[s^{-1},T^{-1}]$, we have $\pd_{R[s^{-1},T^{-1}]}M[T^{-1}] \leq n-1$. Applying \cref{countsubset} to the $(n-1)$-th syzygy $R[s^{-1}]$-module $\Omega_{n-1}(M)$ of $M$, we find a countable multiplicative subset $S$ of $T$ such that $\pd_{R[s^{-1},S^{-1}]}M[S^{-1}] \leq n-1$. Since $S$ is countable, $\pd_R R[s^{-1},S^{-1}] \leq 1$, and therefore $\pd_R M[S^{-1}] \leq n$.

By the assumption that $\Pcal_k(R)$ is of finite type and because $n \leq k$,  $M[S^{-1}] \in \Acal_k(R)$, and therefore also $M[S^{-1}] \otimes_RR_\pp \cong \quotmod{R_\pp[s^{-1}]}{\mathbf{a}R_\pp[s^{-1}]} \in \Acal_k(R_\pp) $.  However, since $\depth(R_\pp) = \findim(R_\pp)$, we have $\Acal_k(R_\pp)\subseteq \Acal_{\depth (R_\pp)}(R_\pp)$, and so $\pd_{R_\pp} \CustomBig(\quotmod{R_\pp[s^{-1}]}{\mathbf{a}R_\pp[s^{-1}]}\CustomBig) \leq \depth (R_\pp)< n$, a contradiction. 
\end{proof}

\begin{proof}[Proof of \cref{main-A}]
Follows by combining \cref{Skimpliesft} and \cref{ftimpliesSk}. The second sentence follows immediately from the fact that $R$ is Cohen--Macaulay if and only if it satisfies $(S_k)$ for all $k \geq 0$.
\end{proof}

\begin{rmk}
While Serre's condition $(S_k)$ is clearly a local property (i.e., it can be checked on the stalk rings $R_\pp$, $\pp \in \Spec R$), it is not straightforward to see that the finite type of $\Pcal_k$ is a local property. Similarly, while $(S_{k+1})$ clearly implies $(S_k)$, it is not straightforward to see that the finite type of $\Pcal_{k+1}$ implies the finite type of $\Pcal_k$. However, \cref{main-A} implies both these features for a commutative noetherian ring for any $k>0$.
\end{rmk}

\section{A Govorov-Lazard-type theorem for $\Fcal_k$}\label{s:lazard}
In this section, we prove \cref{main-B}.

The classical Govorov-Lazard theorem states that every flat module is a direct limit of finitely generated free modules, in symbols: $\Fcal_0 = \varinjlim \Pcal_0^{< \aleph_0}$. In this section we investigate a higher version of this theorem, that is, when every module of flat dimension at most $k$ is a direct limit of strongly finitely presented modules of projective dimension at most $k$. In symbols, we ask when the following equality holds for some $k>0$: $\Fcal_k = \varinjlim \Pcal_k^{< \aleph_0}$. 

We note that this differs from the condition that the modules of flat dimension at most $k$ are the direct limit closure of the modules of projective dimension at most $k$, $\Fcal_k = \varinjlim \Pcal_k$, which is slightly weaker. In fact, it is not even known if $\varinjlim \Pcal_k $ is closed under direct limits, while $\varinjlim \Pcal_k^{< \aleph_0}$ always is, see \cref{ss:injlim}. 

However, things simplify for $k=1$. In the case $k=1$, over any ring $\varinjlim \Pcal_1 = \varinjlim \Pcal_1^{< \aleph_0}$ holds, which is a direct consequence of the fact that $ \Pcal_1 \subseteq \varinjlim \Pcal_1^{< \aleph_0}$. Thus, the Govorov-Lazard theorem was shown to hold over commutative domains for the case $k=1$ in the paper \cite{AHT04}. The following is an example of a commutative noetherian ring where $\Fcal_k \supsetneq \varinjlim \Pcal_k \supsetneq  \varinjlim \Pcal_k^{< \aleph_0}$.

\begin{ex}
Consider the ring $\quotmod{K[x,y,z,w]}{\langle x^2, xy, xz, xw \rangle}$ over a field $K$ with maximal ideal $ \mm \vcentcolon =\langle \bar x, \bar y, \bar z, \bar w \rangle$, and let $R $ be the localisation of this ring at $\mm$. Then the dimension of $R$ is 3, and it has depth 0. Therefore, $\Pcal_1 \subseteq \varinjlim \Pcal_1^{< \aleph_0} = \varinjlim \Pcal_0^{< \aleph_0}= \Fcal_0$ and it follows that also $\Pcal_2 \subseteq \Fcal_1$.  In particular, $\varinjlim \Pcal_2^{< \aleph_0} =  \varinjlim \Pcal_0^{< \aleph_0}= \Fcal_0$ and $ \varinjlim \Pcal_2 \subseteq\Fcal_1$.

If additionally $K$ is countable, $\Fcal_i \subseteq \Pcal_{i+1}$  for $i = 0,1,2$ by \cite[Corollaire (3.3.2)]{RG71}. In particular $\Fcal_2 \supsetneq \Fcal_1= \varinjlim \Pcal_2 \supsetneq  \varinjlim \Pcal_k^{< \aleph_0} = \Fcal_0$ for all $k \geq 0$, since in particular $\Findim(R) = 3$ and $\WFindim(R)=2$.
\end{ex}

\begin{lemma}\label{C2FdimQ}
Let $R$ be a commutative noetherian ring with the property $(C_{2})$. If $Q$ denotes the classical ring of quotients of $R$, then $\Findim(Q) \leq1$ and $\WFindim(Q) =0$.
\end{lemma}
\begin{proof}
It is enough to see that every associated prime of $R$, and thus of $Q$, has height at most $1$. It then follows that $\Kdim(Q)=1$ as every proper ideal of $Q$ is contained in an associated prime of $Q$ as all non-unit elements of $Q$ are zero-divisors. 

Take an associated prime $\pp$ of $R$. Then $\depth (R_\pp)=0$, and by the assumption $(C_2)$, we have the following.
\[
0= \depth(R_\pp) \geq \min\{\height(\pp)-1, 2\}
\]
Therefore, $\height(\pp)\leq 1$, as required. 

As $\Findim(Q)\leq 1$, $\Fcal_1(Q) = \Pcal_1(Q)$ by \cref{ss:RG}, we conclude that $\WFindim(Q) =0$ from \cite[Corollary 6.8]{BH09}.
\end{proof}

The following proofs follow a similar idea to the proof of \cref{indstepSk} with some adjustments.

\begin{lemma}\label{indstepCk}
Let $R$ be a commutative noetherian ring and let $S$ be a multiplicative subset consisting of regular elements of $R$. Take $M \in \Fcal_k(R)$ to be an $S$-divisible module for $k>0$. 
\begin{itemize}
\item[(i)] $M[s] \in \Fcal_{k-1}(\quotmod{R}{sR})$ for each $s \in S$
\item[(ii)] $\Gamma_S(M)$ is a direct limit of modules of the form $M[s], s \in S$.
\end{itemize}
\end{lemma}

\begin{proof}
Let $M \in \Fcal_k(R)$ be an $S$-divisible module and $F^\bullet \to M \to 0$ a flat presentation of $M$ of length $k$. Fix the first yoke $\Upsilon_1(M) = \mathsf{Im}(F_1 \to F_0)$ with respect to the flat presentation of $M$.

(i) Fix $s \in S$. Apply the functor $(- \otimes_R \quotmod{R}{sR})$ to $F^\bullet \to M \to 0$. As $\pd_R(\quotmod{R}{sR}) \leq 1$, $\Tor_i^R(M, \quotmod{R}{sR}) = 0$ for all $i>1$. Next, $\Tor_1^R(\Upsilon_1(M), \quotmod{R}{sR})=0$ as $\Upsilon_1(M)$ is torsion-free, being a submodule of an $R$-flat module. Additionally, by the assumption that $M$ is $S$-divisible, $M \otimes_R \quotmod{R}{sR}=0$. Therefore, we find the following two exact sequences both as $R$-modules and as $\quotmod{R}{sR}$-modules.

\[
\begin{tikzcd}[cramped, sep=small]
0 \arrow[r] & F_k \otimes_R \quotmod{R}{sR} \arrow[r] &  \cdots \arrow[r] &   F_1 \otimes_R \quotmod{R}{sR} \arrow[r] &\Upsilon_1(M) \otimes_R \quotmod{R}{sR} \arrow[r] &0
\end{tikzcd}\\
\]
\[
\begin{tikzcd}[cramped, sep=small]
0 \arrow[r] & \Tor_1^R(M, \quotmod{R}{sR}) \ar[r]& \Upsilon_1(M) \otimes_R \quotmod{R}{sR} \arrow[r] & F_0 \otimes_R \quotmod{R}{sR} \arrow[r] &  0
\end{tikzcd}
\]
%Both sequences are also exact sequences in $\Mod \quotmod{R}{sR}$. 

The first exact sequence is a $\Fcal_0(\quotmod{R}{sR})$-presentation of $\Upsilon_1(M) \otimes_R \quotmod{R}{sR}$, so $\Upsilon_1(M) \otimes_R \quotmod{R}{sR} \in \Fcal_{k-1}(\quotmod{R}{sR})$. 
Additionally, the second exact sequence is pure in $\Mod \quotmod{R}{sR}$ as $F_0 \otimes_R \quotmod{R}{sR}$ is flat in $\Mod \quotmod{R}{sR}$, therefore we also have $M[s] \cong \Tor_1^R(M, \quotmod{R}{sR}) \in \Fcal_{k-1}(\quotmod{R}{sR})$ for every $s \in S$. 

(ii) Since $\Tor^R_1(M,-)$ commutes with direct limits, we find the following, where $R[S^{-1}]$ is the localisation of $R$ at $S$, and $\quotmod{R[S^{-1}]}{R}$ is a direct limit of the modules in $\{R/sR\}_{s \in S}$, see end of \cref{ss:notationmultsubsets}. 
\begin{align*}
\Gamma_S(M) 	&\cong \Tor^R_1(M, \quotmod{R[S^{-1}]}{R}) \\
			&\cong \Tor^R_1(M, \varinjlim_{s \in S}(\quotmod{R}{sR}))\\
			&\cong \varinjlim_{s \in S} \Tor^R_1(M, \quotmod{R}{sR})\\
			&\cong \varinjlim_{s \in S} M[s]\\
\end{align*}

\end{proof}

\begin{rmk}
Unlike in the proof of \cref{indstepSk}, there is no size restriction on the size of the multiplicative subset $S$ in \cref{indstepCk}. Therefore, in particular there is no restriction on the size of the direct limit used to construct $\quotmod{R[S^{-1}]}{R}$, and the modules $R[S^{-1}], \quotmod{R[S^{-1}]}{R}$ are not necessarily of projective dimension one. However, this does not pose an issue in the proof of \cref{indstepCk}. 
\end{rmk}

\begin{prop}\label{Bimp1}
Let $R$ be a commutative noetherian ring and suppose $(C_{k+1})$ holds. Then $\Fcal_k = \varinjlim \Pcal_k^{< \aleph_0}$.

\end{prop}

\begin{proof}
The inclusion $\varinjlim \Pcal_k^{< \aleph_0}\subseteq \Fcal_k $ is clear. The converse will be shown by induction on $k$. For the base step $k=0$, note that $\Fcal_0 = \varinjlim \Pcal_0^{< \aleph_0}$ holds for every ring by the Govorov-Lazard theorem. 

Let $F \in \Fcal_k$, and assume that $(C_{i+1})$ implies $\varinjlim \Pcal_i^{< \aleph_0}= \Fcal_i $ for all natural numbers $i<k$. Let $\Sigma$ be the multiplicative subset of $R$ consisting of all regular elements. We first claim that it is sufficient to show that any $\Sigma$-divisible module in $\Fcal_k$ is in $\varinjlim \Pcal_k^{< \aleph_0}$. In fact, take a special $\Dcal_\Sigma$-preenvelope of some $F \in \Fcal_k$,
\[
\begin{tikzcd}%[cramped, sep=small]
0 \arrow[r] &F \arrow[r] &   D \arrow[r] &A \arrow[r] &0.
\end{tikzcd}
\]
Then $A$ is in $\Pcal_1$, $\Pcal_1 \subseteq \varinjlim \Pcal_1^{< \aleph_0}$, and $D \in \Fcal_k$. It follows that $F \in \varinjlim \Pcal_k^{< \aleph_0} $ if and only if $D \in \varinjlim \Pcal_k^{< \aleph_0} $ (see \cref{ss:injlim}), therefore we only need to restrict to showing that $\Fcal_k(R)\cap \Dcal_\Sigma \subseteq\varinjlim \Pcal_k^{< \aleph_0}$. 

Let $M \in \Fcal_k(R)\cap \Dcal_\Sigma$, and apply $(M \otimes_R-)$ to the short exact sequence $0 \to R \to Q \to \quotmod{Q}{R} \to 0$ where $Q$ is the classical ring of quotients of $R$,
\[
\begin{tikzcd}%[cramped, sep=small]
0 \arrow[r] &\Gamma_\Sigma(M) \arrow[r] &   M \arrow[r] &M \otimes_RQ \arrow[r] &0.
\end{tikzcd}
\]
The above sequence is exact by the assumption that $M$ is $\Sigma$-divisible. We will show that the two outer modules are in $\varinjlim \Pcal_k^{< \aleph_0}$.

 We first claim that $M \otimes_RQ \in \varinjlim \Pcal_k^{< \aleph_0}(R)$. By \cref{C2FdimQ}, $\WFindim(Q)=0$, so $M \otimes_RQ \in \Fcal_k(Q) = \Fcal_0(Q)$. As all flat modules over $Q$ are flat over $R$, $M \otimes_RQ \in \Fcal_0(R) = \varinjlim \Pcal_0^{< \aleph_0}(R) \subseteq \varinjlim \Pcal_k^{< \aleph_0}$.

We now claim that $\Gamma_\Sigma(M) \in \varinjlim \Pcal_k^{< \aleph_0}(R)$. By \cref{indstepCk} (ii), \cref{serrequotientring} and the induction hypothesis, $M[s] \in \Fcal_{k-1}(\quotmod{R}{sR}) = \varinjlim \Pcal_{k-1}^{< \aleph_0}(\quotmod{R}{sR})$ for any $s \in \Sigma$, and by the change of rings theorem (see \cite[Theorem 4.3.3]{Wei94}), $\Pcal_{k-1}^{< \aleph_0}(\quotmod{R}{sR}) \subseteq\Pcal_{k}^{< \aleph_0}(R)$, so we conclude that $M[s] \in \varinjlim \Pcal_k^{< \aleph_0}(R)$. By \cref{indstepCk} (ii), $\Gamma_\Sigma(M) $ is a direct limit of the modules $M[s]$ with $s \in \Sigma$, so since $\varinjlim \Pcal_k^{< \aleph_0}(R)$ is closed under direct limits, see \cref{ss:injlim}, $\Gamma_\Sigma(M) \in \varinjlim \Pcal_k^{< \aleph_0}(R)$.

\end{proof}

\begin{prop}\label{Bimp2}
Let $R$ be a commutative noetherian ring and $\Fcal_k = \varinjlim \Pcal_k^{< \aleph_0}$ for $k \geq 0$. Then $(C_{k+1})$ holds. 
\end{prop}

\begin{proof}
Suppose that $\Fcal_k(R) = \varinjlim \Pcal_k^{< \aleph_0}(R)$ and take a prime ideal $\pp$ such that $\depth(R_\pp) < \height(\pp)-1$. We will show that $\depth (R_\pp) \geq k$.

We have the following inclusions.
\begin{align*}
\Fcal_k(R_\pp) 	& \subseteq \Fcal_k(R) \cap \Mod R_\pp \\
			&= \varinjlim \Pcal_k^{< \aleph_0}(R) \cap \Mod R_\pp \\
			&\subseteq \varinjlim \Pcal_k^{< \aleph_0}(R_\pp) \\
			& \subseteq \varinjlim \Pcal_{\depth (\pp)}^{< \aleph_0}(R_\pp)  \\
			&\subseteq \Fcal_{\depth (\pp)}(R_\pp)  
\end{align*}
The first inclusion follows since flat $R_\pp$-modules are flat as $R$-modules. The second follows by the assumption. The third follows since the tensor product $(- \otimes_RR_\pp)$ commutes with direct limits, and sends finitely presented $R$-modules of finite $R$-projective dimension to finitely presented $R_\pp$-modules of finite $R_\pp$-projective dimension. The fourth follows since $\depth(R_\pp) = \findim(R_\pp)$, so $\Pcal_{i}^{< \aleph_0}(R_\pp)\subseteq \Pcal_{\depth(R_\pp)}^{< \aleph_0}(R_\pp)$ for all natural numbers $i$. The final inclusion is clear.

Moreover, by \cite[Corollary 5.3]{B62} and the assumption, $\WFindim(R_\pp)\geq \dim(R_\pp)-1 = \height(\pp)-1 > \depth(R_\pp)$ 
so $\Fcal_{\depth (R_\pp)+1}(R_\pp) \supsetneq \Fcal_{\depth (R_\pp)}(R_\pp)$ and we can conclude that $k \leq \depth(R_\pp)$, as required.
\end{proof}

\begin{proof}[Proof of \cref{main-B}]
	Follows by combining \cref{Bimp1} and \cref{Bimp2}. The second sentence follows immediately from the fact that $R$ is almost Cohen--Macaulay if and only if it satisfies $(C_k)$ for all $k > 0$.
\end{proof}

\begin{rmk}
	Similarly for the case of finite type of $(\Pcal_k,\Pcal_k\Perp{})$, \cref{main-B} shows that the validity of $k$-dimensional Govorov-Lazard Theorem is a local property, and that it implies the $(k-1)$-dimensional Govorov-Lazard Theorem.
\end{rmk}

\bibliographystyle{amsalpha}
\bibliography{bibitems}

\providecommand{\bysame}{\leavevmode\hbox to3em{\hrulefill}\thinspace}
\providecommand{\MR}{\relax\ifhmode\unskip\space\fi MR }
% \MRhref is called by the amsart/book/proc definition of \MR.
\providecommand{\MRhref}[2]{%
  \href{http://www.ams.org/mathscinet-getitem?mr=#1}{#2}
}
\providecommand{\href}[2]{#2}
\begin{thebibliography}{AHP{\v{S}}T14}

\bibitem[AEJO01]{AEJO}
S.~Tempest Aldrich, Edgar~E. Enochs, Overtoun M.~G. Jenda, and Luis Oyonarte,
  \emph{Envelopes and covers by modules of finite injective and projective
  dimensions}, J. Algebra \textbf{242} (2001), no.~2, 447--459. \MR{1848954}

\bibitem[AHHT05]{AHHT05}
Lidia Angeleri~H\"{u}gel, Dolors Herbera, and Jan Trlifaj, \emph{Divisible
  modules and localization}, J. Algebra \textbf{294} (2005), no.~2, 519--551.
  \MR{2183363}

\bibitem[AHP{\v{S}}T14]{AHPS14}
Lidia Angeleri~H\"{u}gel, David Posp\'{\i}\v{s}il, Jan
  {\v{S}}\v{t}ov\'{\i}\v{c}ek, and Jan Trlifaj, \emph{Tilting, cotilting, and
  spectra of commutative {N}oetherian rings}, Trans. Amer. Math. Soc.
  \textbf{366} (2014), no.~7, 3487--3517.

\bibitem[AHT04]{AHT04}
Lidia Angeleri~H\"{u}gel and Jan Trlifaj, \emph{Direct limits of modules of
  finite projective dimension}, Rings, modules, algebras, and abelian groups,
  Lecture Notes in Pure and Appl. Math., vol. 236, Dekker, New York, 2004,
  pp.~27--44. \MR{2050699}

\bibitem[Bas62]{B62}
Hyman Bass, \emph{Injective dimension in {N}oetherian rings}, Transactions of
  the American Mathematical Society \textbf{102} (1962), no.~1, 18--29.

\bibitem[BH98]{BH98}
Winfried Bruns and J\"{u}rgen Herzog, \emph{Cohen-{M}acaulay rings}, revised
  ed., Cambridge Studies in Advanced Mathematics, vol.~39, Cambridge University
  Press, Cambridge, 1998.

\bibitem[BH08]{BH08}
Silvana Bazzoni and Dolors Herbera, \emph{One dimensional tilting modules are
  of finite type}, Algebr. Represent. Theory \textbf{11} (2008), no.~1, 43--61.
  \MR{2369100}

\bibitem[BH09]{BH09}
\bysame, \emph{Cotorsion pairs generated by modules of bounded projective
  dimension}, Israel J. Math. \textbf{174} (2009), 119--160. \MR{2581211}

\bibitem[B{\v{S}}07]{BS07}
Silvana Bazzoni and Jan {\v{S}}ťov{\'\i}{\v{c}}ek, \emph{All tilting modules
  are of finite type}, Proceedings of the American Mathematical Society
  \textbf{135} (2007), no.~12, 3771--3781.

\bibitem[CFF02]{CFF02}
Lars~Winther Christensen, Hans-Bj{\o}rn Foxby, and Anders Frankild,
  \emph{Restricted homological dimensions and {C}ohen--{M}acaulayness}, Journal
  of Algebra \textbf{251} (2002), no.~1, 479--502.

\bibitem[ET01]{ET01}
Paul~C. Eklof and Jan Trlifaj, \emph{How to make {E}xt vanish}, Bull. London
  Math. Soc. \textbf{33} (2001), no.~1, 41--51. \MR{1798574}

\bibitem[GT12]{GT12}
R\"{u}diger G\"{o}bel and Jan Trlifaj, \emph{Approximations and endomorphism
  algebras of modules\textup{:} {V}olume 1 -- {A}pproximations}, second revised
  and extended ed., De Gruyter Expositions in Mathematics, vol.~41, Walter de
  Gruyter GmbH \& Co. KG, Berlin, 2012.

\bibitem[Han98]{H98}
Yang Han, \emph{{$D$} rings}, Acta Math. Sinica (Chinese Ser.) \textbf{41}
  (1998), no.~5, 1047--1052. \MR{1681796}

\bibitem[HLG23]{HLG23}
Michal Hrbek and Giovanna Le~Gros, \emph{Restricted injective dimensions over
  cohen-macaulay rings}, arXiv preprint arXiv:arXiv:2305.18884 (2023).

\bibitem[Hol19]{H19}
Brent Holmes, \emph{A generalized {S}erre's condition}, Comm. Algebra
  \textbf{47} (2019), no.~7, 2689--2701. \MR{3960357}

\bibitem[HPS22]{HPS}
Michal Hrbek, Leonid Positselski, and Alexander Sl\'{a}vik, \emph{Countably
  generated flat modules are quite flat}, J. Commut. Algebra \textbf{14}
  (2022), no.~1, 37--54. \MR{4430700}

\bibitem[H{\v{S}}20]{HS20}
Michal Hrbek and Jan {\v{S}}ťov\'{\i}\v{c}ek, \emph{Tilting classes over
  commutative rings}, Forum Math. \textbf{32} (2020), no.~1, 235--267.
  \MR{4048464}

\bibitem[Ion15]{Ion}
Cristodor Ionescu, \emph{More properties of almost {C}ohen-{M}acaulay rings},
  J. Commut. Algebra \textbf{7} (2015), no.~3, 363--372. \MR{3433987}

\bibitem[Jen70]{Jen70}
C.~U. Jensen, \emph{On the vanishing of {$\varprojlim\sp{(i)}$}}, J. Algebra
  \textbf{15} (1970), 151--166. \MR{260839}

\bibitem[Kan01]{K01}
Ming-Chang Kang, \emph{Almost {C}ohen-{M}acaulay modules}, Comm. Algebra
  \textbf{29} (2001), no.~2, 781--787. \MR{1841999}

\bibitem[Len83]{Lenz83}
H.~Lenzing, \emph{Homological transfer from finitely presented to infinite
  modules}, Abelian group theory ({H}onolulu, {H}awaii, 1983), Lecture Notes in
  Math., vol. 1006, Springer, Berlin, 1983, pp.~734--761. \MR{722664}

\bibitem[RG71]{RG71}
Michel Raynaud and Laurent Gruson, \emph{Crit\`eres de platitude et de
  projectivit\'{e}. {T}echniques de ``platification'' d'un module}, Invent.
  Math. \textbf{13} (1971), 1--89.

\bibitem[Ste75]{SS75}
Bo~Stenstr{\"o}m, \emph{Rings and modules of quotients}, Rings of Quotients,
  Springer, 1975, pp.~195--212.

\bibitem[Wei94]{Wei94}
Charles~A. Weibel, \emph{An introduction to homological algebra}, Cambridge
  Studies in Advanced Mathematics, vol.~38, Cambridge University Press,
  Cambridge, 1994. \MR{1269324}

\end{thebibliography}
\end{document}